\documentclass[11pt,centertags]{amsart}
\usepackage{amscd}
\usepackage{easybmat}
\usepackage{mathrsfs}
\usepackage{amsfonts}
\usepackage{color}
\usepackage{pifont}
\usepackage{upgreek}
\usepackage{bm}
\usepackage{hyperref}
\usepackage{shorttoc}
\usepackage{amsmath,amstext,amsthm,a4,amssymb,amscd}
\usepackage[mathscr]{eucal}
\usepackage{mathrsfs}
\usepackage{epsf}
\numberwithin{equation}{section}

\def\p{\partial}
\def\o{\overline}
\def\b{\bar}
\def\mb{\mathbb}
\def\mc{\mathcal}
\def\n{\nabla}

\newtheorem{thm}{Theorem}[section]
\newtheorem{lemma}[thm]{Lemma}
\newtheorem{prop}[thm]{Proposition}
\newtheorem{cor}[thm]{Corollary}
\theoremstyle{definition}
\newtheorem{rem}[thm]{Remark}

\newtheorem{ex}[thm]{Example}
\theoremstyle{definition}
\newtheorem{defn}[thm]{Definition}
\newcommand{\comment}[1]{}
\usepackage{fancyhdr}
\begin{document}

\title{Holomorphic sectional curvature of complex Finsler manifolds}

\author[Xueyuan Wan]{Xueyuan Wan}


\address{Xueyuan Wan: Mathematical Sciences, Chalmers University of Technology, 41296 Gothenburg, Sweden}
\email{xwan@chalmers.se}

\thanks{Partially supported by NSFC (Grant No. 11221091,11571184)}

\begin{abstract}
In this paper, we get an inequality in terms of holomorphic sectional curvature of complex Finsler metrics. As applications, we prove a Schwarz Lemma from a complete Riemannian manifold to a complex Finsler manifold. We also show that a strongly pseudoconvex complex Finsler manifold with  semi-positive but not identically zero holomorphic sectional curvature has negative Kodaira dimension under an extra condition.
 \end{abstract}
\maketitle
\tableofcontents


\section*{Introduction} \label{s0}

In this paper, we study the holomorphic sectional curvature of complex Finsler manifolds (see Definition \ref{defn1}). For a general complex manifold, there is a natural and intrinsic Finsler pseudo-metric, i.e . Kobayashi metric $k_M$ (see Definition \ref{defn2}), which is the maximum pseudo-metric among the pseudo metrics satisfying the decreasing property. This metric defines a Kobayashi pseudo-distance, and a complex manifold is called Kobayashi hyperbolic if the Kobayashi pseudo-distance is a distance in common sense.  It is known that a complex Finsler manifold is Kobayashi hyperbolic if its holomorphic sectional curvature is bounded from above by a negative constant. As an application, the moduli space of canonically polarized complex manifolds is Kobayshi hyperbolic \cite{To, Sch1}.

In hyperbolic geometry a conjecture of Kobayashi asserts that the canonical bundle is ample if the manifold is hyperbolic \cite[p. 370]{Ko5}.
 Wu and Yau \cite{Wu} proved the ampleness of the canonical bundle for a projective manifold admitting a K\"ahler metric with negative holomorphic sectional curvature. Later on, V. Tosatti and X. Yang \cite{Tosatti} proved this result without projective condition and asked if it still true for a compact Hermitian manifold. In \cite{Yau4}, Yau conjectured that an algebraic manifold is of general type if and only if it admits a complex Finsler metric with strongly negative holomorphic sectional curvature which may be degenerate along a subvariety. For the case of quasi-negative holomorphic sectional curvature, the ampleness of the  canonical bundle also had been  proved in \cite{Div1, Wu1}. In the proofs of above results, they  essentially  used Yau's Schwarz Lemma \cite{Yau1}. This is also a motivation for us to study the Schwarz Lemma in the case of complex Finsler manifolds.

 \begin{prop}\label{0.1}
Let $M$ be a complex manifold of dimension $n$, $G_1$ and $G_2$ be two complex Finsler metrics on $M$. For any  point $(z,[v])\in P(TM)$, one has
\begin{align}
\frac{2}{G_2}\sup_{\pi_*(X)=v}\left\{\p\b{\p}\log\frac{G_1}{G_2}(X,\o{X})\right\}\geq K_{G_2}-\frac{G_1}{G_2}K_{G_1},
\end{align}
where $\pi_*: T(TM)\to TM$ is the differential of $\pi: TM\to M$.
\end{prop}
A simple maximum principle argument immediately gives a direct proof for the Schwarz Lemma of \cite[Theorem 4.1]{Shen}. For the case of a complete Riemann surface $(M,G_1)$, by using Proposition \ref{0.1},  we obtain the following generalization of Yau's Schwarz Lemma (\cite[Theorem 2']{Yau1}).
 \begin{thm}\label{thm1}
Let $(M,G_1)$ be a complete Riemann surface with curvature bounded from below by a constant $K_1$. Let $(N, G_2)$ be another Finsler manifold with holomorphic sectional curvature bounded from above by a negative constant $K_2$. Then for any holomorphic map $f$ from $M$ to $N$,
\begin{align}\label{1.4}
\frac{f^*G_2}{G_1}\leq \frac{K_1}{K_2}.
\end{align}
\end{thm}

If $(M,G_1)$ is  a unit disc with a Poincar\'e metric (which is complete) and by the definition of Kobayashi metric, one has
$4k^2_N\geq -K_2G_1$. This implies that $N$ is Kobayashi hyperbolic by definition.

For a compact K\"ahler manifold with positive holomorphic sectional curvature, Yau \cite[Problem 67]{Yau3} asked if the manifold has negative Kodaira dimension. X. Yang \cite{Xiao} gave a affirmative answer to this problem. More precisely, if $(M,\omega)$ is a Hermitian manifold with semi-positive but not identically zero holomorphic sectional curvature, then the Kodaria dimension $\kappa(M)=-\infty$.
Naturally, one may ask whether this result holds for a strongly pseudoconvex complex Finsler manifold, namely whether for a compact strongly pseudoconvex Finsler manifold $(M,G)$ with semi-positive but not identically zero holomorphic sectional curvature one has that $\kappa(M)=-\infty$.  

For any strongly pseudocovex complex Finsler metric $G$, there is a canonical (Finslerian) tensor
\begin{align}
\hat{G}:=\frac{\p^2 G}{\p v^i\p v^j}\delta v^i\otimes \delta v^j\in A^0((TM)^o, \mc{V}^*\otimes\mc{V}^*).
\end{align} 
Here $(TM)^o$ denotes the set of all non zero holomorphic tangent vectors of $M$, $\{\delta v^i\}_{i=1}^n$ is a local holomorphic frame of $\mc{V}^*$ (for the definitions of $\delta v^i$ and $\mc{V}^*$ see (\ref{H})).  By taking covariant derivative of the Finslerian tensor $\hat{G}$ along the anti-holomorphic direction $\o{\hat{P}}:=\o{v^i\frac{\delta}{\delta z^i}}$, we get 
\begin{align}
	\b{\p}\hat{G}(\o{\hat{P}})\in A^0((TM)^o, \mc{V}^*\otimes\mc{V}^*).
\end{align}

By using Berndtsson's curvature formula to the cotangent bundle, we prove that
\begin{thm}
	Let $(M,G)$ be a compact strongly pseudoconvex complex Finsler manifold with semi-positive but not identically zero holomorphic sectional curvature, and satisfy $\b{\p}\hat{G}(\o{\hat{P}})=0$, then $\kappa(M)=-\infty$.
\end{thm}
\begin{rem}
From the definition of $\hat{P}$ and $\hat{G}$ (see (\ref{hatP}, \ref{hatG}) in Section \ref{section1}), if $G$ comes from a Hermitian metric, the Finslerian tensor $\hat{G}$ is identically trivial, so that any Hermitian metric satisfies the condition $\b{\p}\hat{G}(\o{\hat{P}})=0$. Moreover, there are also many non Hermitian strongly pseudoconvex complex Finsler metrics which satisfy the condition (see Example \ref{ex}).
\end{rem}

\

{\bf Acknowledgements.} The author would like to thank Professor Kefeng Liu and Professor Huitao Feng for their suggestions on the study of holomorphic sectional curvature of complex Finsler manifolds. And thank Professor Bo Berndtsson and Professor Xiaokui Yang for many helpful discussions. The author would like to thank the anonymous referee for valuable comments which helped to improve the paper.


\section{Holomorphic sectional curvature of complex Finsler manifolds}\label{section1}

In this section, we shall fix notation and recall some basic definitions and facts on complex Finsler manifolds. For more details we refer to  \cite{AP, Aikou1, Aikou, Cao-Wong, FLW, Ko1, Wong1}.

Let $M$ be a complex manifold of dimension $n$, and let $\pi:TM\to M$ be the holomorphic tangent bundle of $M$.  Let $z=(z^1,\cdots, z^n)$ be a local coordinate system  in $M$, and let $\{\frac{\p}{\p z^i}\}_{1\leq i\leq n}$ denote the corresponding natural frame of $TM$. So any element in $TM$ can be written as
$$v=v^i\frac{\p}{\p z^i}\in TM,$$
where we adopt the summation convention of Einstein. In this way, one gets a local coordinate system on the complex manifold $TM$: 
\begin{align}\label{coor}
(z;v)=(z^1,\cdots,z^n; v^1,\cdots, v^n).
\end{align}

\begin{defn}[\cite{Ko1, Ko5}]
A Finsler metric  $G$
on the complex manifold $M$ is a continuous function $G:TM\to\mathbb{R}$ satisfying the following conditions:

\begin{description}
  \item [F1)] $G$ is smooth on $(TM)^o=TM\setminus O$, where $O$ denotes the zero section of $TM$;
  \item[F2)] $G(z,v)\geq 0$ for all $(z,v)\in TM$ with $z\in M$ and $v\in\pi^{-1}(z)$, and $G(z,v)=0$ if and only if $v=0$;
  \item[F3)] $G(z,\lambda v)=|\lambda|^2G(z,v)$ for all $\lambda\in\mathbb{C}$.
\end{description}

Moreover, $G$ is called strongly pseudoconvex if

\begin{description}
  \item[F4)] the Levi form ${\sqrt{-1}}\partial\bar\partial G$ on $(TM)^o$ is positive-definite along fibres $(TM)_z=\pi^{-1}(z)$ for $z\in M$.
\end{description}
$(M,G)$ is called a (strongly pseudoconvex) complex Finsler manifold if $G$ is a (strongly pseudoconvex) complex Finsler metric.
\end{defn}
Clearly, any Hermitian metric on $M$ is naturally a strongly pseudoconvex complex Finsler metric on it.

We write
\begin{align}
&G_i=\frac{\p G}{\p v^i},\quad G_{\b{j}}=\frac{\p G}{\p\b{v}^j},\quad G_{i\b{j}}=\frac{\p^2 G}{\p v^i\p\b{v}^j},\\
&G_{i;j}=\frac{\p^2 G}{\p v^i\p z^{j}},\quad G_{i\b{j};\b{l}}=\frac{\p^3 G}{\p v^i\p\b{v}^j\p\b{z}^\beta},\quad  \text{etc}.,
\end{align}
to denote the differentiation with respect to $v^i,\b{v}^j, z^i,\b{z}^j$, and we denote $(G^{\b{j}i})$ the inverse matrix of $(G_{i\b{j}})$. In the following lemma we collect some useful identities related to a Finsler metric $G$.
\begin{lemma}[\cite{Cao-Wong, Ko1}]\label{1.111} The following identities hold for any $(z,v)\in E^o$, $\lambda\in \mathbb{C}$:
\begin{align}
G_i(z,\lambda v)=\bar\lambda G_i(z,v),\quad G_{i\bar j}(z,\lambda v)=G_{i\bar j}(z,v)=\bar G_{j\bar i}(z,v);
\end{align}
\begin{align}
G_i(z,v)v^i=G_{\bar j}(z,v)\bar v^j=G_{i\bar j}(z,v)v^i\bar v^j=G(z,v);
\end{align}
\begin{align}
G_{ij}(z,v)v^i=G_{i\bar j k}(z,v)v^i=G_{i\bar j\bar k}(z,v)\bar v^j=0.
\end{align}
\end{lemma}

Let $\Delta=\{w\in\mb{C}||w|<1\}$ be a unit disc. For any holomorphic map $\varphi:\Delta\to M$, one can define a conformal metric on the disc $\Delta$ by
\begin{align}
ds^2=\varphi^*G dw\otimes d\b{w},	
\end{align}
where $(\varphi^*G)(u\frac{\p}{\p w}):=G(\varphi_*(u\frac{\p}{\p w}))$.
The Gaussian curvature $K_{\varphi^*G}$ of $\varphi^*G$ is given by
\begin{align}
K_{\varphi^*G}=-\frac{2}{\varphi^*G}\frac{\p^2\log \varphi^*G}{\p w\p\b{w}}.	
\end{align}
Then one can define the holomorphic sectional curvature of $(M,G)$ as follows.
\begin{defn}[\cite{AP1, Aikou1}]\label{defn1}
The holomorphic sectional curvature $K_G(z,[v])$ of $G$ at $(z,[v])\in P(TM):=(TM)^o/\mb{C}^*$ is defined by
\begin{align}
K_G(z,[v]):=\sup_{\varphi}\{K_{\varphi^*G}(0)\},	
\end{align}
	where the supremum is taken over all the holomorphic maps $\varphi:\Delta\to M$ satisfying $\varphi(0)=z$, $\varphi'(0)=\lambda v$ for some $\lambda\in \mb{C}^*$.
\end{defn}

If $G$ is a strongly pseudoconvex complex Finsler metric on $M$, then there is a canonical h-v decomposition of the holomorphic tangent bundle $T(TM)^o$ of $(TM)^o$ (see \cite[\S 5]{Cao-Wong} or \cite[\S 1]{FLW}).
\begin{align}
T(TM)^o=\mc{H}\oplus \mc{V}.
\end{align}
In terms of local coordinates,
\begin{align}
\mc{H}=\text{span}_{\mb{C}}\left\{\frac{\delta}{\delta z^i}=\frac{\p}{\p z^i}-G_{\b{j};i}G^{\b{j}k}\frac{\p}{\p v^k}, 1\leq i\leq n\right\},\quad \mc{V}=\text{span}_{\mb{C}}\left\{\frac{\p}{\p v^i}, 1\leq i\leq n\right\}.
\end{align}
Moreover, the dual bundle $T^*(TM)^o$ also has a smooth h-v decomposition $T^*(TM)^o=\mc{H}^*\oplus\mc{V}^*$ with
\begin{align}\label{H}
\mc{H}^*=\text{span}_{\mb{C}}\{dz^{i}, 1\leq i\leq n\},\quad \mc{V}^*=\text{span}_{\mb{C}}\{\delta v^i=dv^i+G^{\b{j}i}G_{\b{j};k}dz^k,\quad 1\leq i\leq n\}.
\end{align}
With respect to the h-v decomposition (\ref{H}), the $(1,1)$-form $\frac{\sqrt{-1}}{2\pi}\p\b{\p}\log G$ has the following decomposition.
\begin{lemma}[\cite{Ko1, Aikou}]\label{lemma1}
Let $G$ be a strongly pseudoconvex complex Finsler metric on $M$. One has
\begin{align}
\frac{\sqrt{-1}}{2\pi}\p\b{\p}\log G=-\frac{\Psi}{2\pi}+\omega_{FS},
\end{align}
where we denote
\begin{align}
\Psi=\sqrt{-1}R_{i\b{j}k\b{l}}\frac{v^i\b{v}^j}{G}dz^k\wedge d\b{z}^l,\quad \omega_{FS}=\frac{\sqrt{-1}}{2\pi}\frac{\p^2 \log G}{\p v^i\p\b{v}^j}\delta v^i\wedge \delta\b{v}^j,
\end{align}
with
$$R_{i\b{j}k\b{l}}=-\frac{\p^2 G_{i\b{j}}}{\p z^k\p\b{z}^l}+G^{\b{q}p}\frac{\p G_{i\b{q}}}{\p z^k}\frac{\p G_{p\b{j}}}{\p\b{z}^l}.$$
\end{lemma}

For the holomorphic vector bundle $T(TM)^o\to (TM)^o$, there are two special smooth vector fields
\begin{align}\label{hatP}
\hat{P}=v^i\frac{\delta}{\delta z^i}\in \mc{H},\quad P=v^i\frac{\p}{\p v^i}\in \mc{V},
\end{align}
which are well-defined. Indeed, for two local coordinate neighborhoods $(U_{\alpha}, \{z_{\alpha}^i\}), (U_{\beta}, \{z^j_{\beta}\})$ of $M$ with $U_{\alpha}\cap U_{\beta}\neq \emptyset$, one has
\begin{align}\label{1.1}
v^i_{\alpha}=\frac{\p z^i_{\alpha}}{\p z^j_{\beta}}v^j_{\beta}, \quad \frac{\delta}{\delta z^i_{\alpha}}=\frac{\p z^j_{\beta}}{\p z^i_{\alpha}}\frac{\delta}{\delta z^j_{\beta}}.
\end{align}

For a strongly pseudoconvex complex Finsler metric $G$, one can also define the holomorphic sectional curvature by
\begin{align}\label{defn22}
K_G=2R_{i\b{j}k\b{l}}\frac{v^i\b{v}^jv^k\b{v}^l}{G^2}=-\frac{2}{G}(\p\b{\p}\log G)(\hat{P},\o{\hat{P}}).
\end{align}
(see \cite[Formula (6.7)]{Ko1}), which is a  function on projective bundle $P(TM)$. In this case, the two kinds of definitions for holomorphic sectional curvature coincide. More precisely, one has
\begin{prop}[{\cite[Corollary 2.5.4]{AP}, \cite[Proposition 7.2]{Aikou2}}]
If $G$ is a strongly pseudoconvex complex Finsler metric on $M$, then
\begin{align}
\sup_{\varphi}\{K_{\varphi^*G}(0)\}=2R_{i\b{j}k\b{l}}\frac{v^i\b{v}^jv^k\b{v}^l}{G^2},
\end{align}
where the supremum is taken over all the holomorphic maps $\varphi:\Delta\to M$ satisfying $\varphi(0)=z$, $\varphi'(0)=\lambda v$ for some $\lambda\in \mb{C}^*$.
\end{prop}
\begin{proof} For reader's convenient, we give a direct proof here.
For any holomorphic $\varphi:\Delta\to M$ with $\varphi(0)=z$, $\varphi'(0)=\lambda v$ for some $\lambda\in \mb{C}^*$, it induces a holomorphic map
\begin{align}
\varphi_*: T\Delta\to TM\quad \varphi_*(w,u)=\varphi_*(u\frac{\p}{\p w})=u\frac{\p\varphi^i}{\p w}\frac{\p}{\p z^i}=(\varphi(z); u\frac{\p\varphi^i}{\p w}),
\end{align}
which satisfies
\begin{align}
\varphi_*(0,1)=(z,\lambda v).
\end{align}
Similarly, the holomorphic map $\varphi_*$ also induces a holomorphic map $(\varphi_*)_*: T(T\Delta)\to T(TM)$ and
\begin{align}
(\varphi_*)_*(\frac{\p}{\p w}|_{(0,1)})=\lambda v^i\frac{\p}{\p z^i}|_{(z,\lambda v)}+\frac{\p^2\varphi^i}{\p w^2}\frac{\p}{\p v^i}|_{(z,\lambda v)}.
\end{align}
Therefore, one has
\begin{align}
\begin{split}
&\quad \sup_{\varphi}\{K_{\varphi^*G}(0)\}=\sup_{\varphi}\left\{-\frac{2}{G(z,\lambda v)}\left(\frac{\p^2}{\p w\p\b{w}}\log {\varphi}^*G\right)(0)\right\}\\
&=\sup_{\varphi}\left\{-\frac{2}{G(z,\lambda v)}(\varphi_*)^*(\p\b{\p}\log G)\left(\frac{\p}{\p w}|_{(0,1)},\frac{\p}{\p \b{w}}|_{(0,1)}\right)\right\}\\
&=\sup_{\varphi}\left\{-\frac{2}{G(z,\lambda v)}(\p\b{\p}\log G)\left((\varphi_*)_*(\frac{\p}{\p w}|_{(0,1)}),(\varphi_*)_*(\frac{\p}{\p \b{w}}|_{(0,1)})\right)\right\}\\
&=\sup_{\varphi}\left\{-\frac{2}{G(z,\lambda v)}(\p\b{\p}\log G)\left(\lambda v^i\frac{\delta}{\delta v^i}+(\frac{\p^2\varphi^k}{\p w^2}+G_{\b{l};i}G^{\b{l}k}\lambda v^i)\frac{\p}{\p v^k},\o{\lambda v^i\frac{\delta}{\delta v^i}+(\frac{\p^2\varphi^k}{\p w^2}+G_{\b{l};i}G^{\b{l}k}\lambda v^i)\frac{\p}{\p v^k}}\right)\right\},
\end{split}
\end{align}
By Lemma \ref{lemma1} and the property that $\omega_{FS}$ is positive along the fiber of $P(TM)$, taking the holomorphic map $\varphi:\Delta\to M$ with
\begin{align}
\varphi(0)=0, \quad \varphi'(0)=\lambda v, \quad \frac{\p^2\varphi^k}{\p w^2}(0)=-\lambda G_{\b{l};i}G^{\b{l}k} v^i,
\end{align}
one gets that
\begin{align}
	 \sup_{\varphi}\{K_{\varphi^*G}(0)\}=-\frac{2}{G(z,v)}\p\b{\p}\log G(v^i\frac{\delta}{\delta v^i},\o{v^i\frac{\delta}{\delta v^i}})=2R_{i\b{j}k\b{l}}\frac{v^i\b{v}^jv^k\b{v}^l}{G^2}.
\end{align}

\end{proof}

For a given strongly pseudocovex complex Finsler metric $G$, there is a canonical (Finslerian) tensor $\hat{G}$,
\begin{align}\label{hatG}
\hat{G}:=G_{ij}\delta v^i\otimes \delta v^j\in A^0((TM)^o, \mc{V}^*\otimes\mc{V}^*).
\end{align}
From (\ref{1.1}), it is easy to see that $\hat{G}$ is well-defined. Moreover, $\hat{G}=0$ if and only if the strongly pseudoconvex complex Finsler metric $G$ comes from a Hermitian metric. In fact, if $G$ comes from  Hermitian metric $h$ on $M$, i.e. $G=h_{i\b{j}}(z)v^i\b{v}^j$, then
$$G_{ij}=\frac{\p^2 G}{\p v^i\p v^j}=0.$$
So $\hat{G}=0$. Conversely, if $\hat{G}=0$, then
$$G_{i\b{k}j}=\frac{\p^2 G_{ij}}{\p\b{v}^k}=0.$$
By taking conjugation, one gets $G_{i\b{k}\b{j}}=0$. So $G_{i\b{j}}(z,v)$ is independent of the fiber, i.e. $G_{i\bar{j}}(z,v)=G_{i\b{j}}(z)$. Therefore, $G=G_{i\b{j}}(z)v^i\b{v}^j$ comes from a Hermitian metric.

Note that $\mc{V}^*\otimes\mc{V}^*\to (TM)^o$ is a holomorphic vector bundle with a local holomorphic frame $\{\delta v^i\otimes \delta v^j\}_{1\leq i,j\leq n}$. So
\begin{align}
\b{\p}\hat{G}\in A^{0,1}((TM)^o,\mc{V}^*\otimes \mc{V}^*).
\end{align}

In applications, one may consider a class of strongly pseudoconvex complex Finsler metrics $G$, which satisfies
\begin{align}\label{1.2}
\b{\p}\hat{G}(\o{\hat{P}})=0.
\end{align}

In terms of local coordinates, the equation (\ref{1.2}) is equivalent to
\begin{align}
\b{v}^l\frac{\delta}{\delta\b{z}^l}G_{ij}=0.
\end{align}

\begin{ex}\label{ex}
There are many non Hermitian strongly pseudoconvex complex Finsler metrics which satisfying condition (\ref{1.2}). For example, any 
complex Finsler manifold $(M,G)$ \textsl{modeled on a complex Minkowski space} satisfies the condition (\ref{1.2}). In fact, in this case, one has $\frac{\p}{\p\b{v}^j}(G_{\b{l};i}G^{\b{l}k})=0$ (see \cite[Definition 2.7, Proposition 2.10]{Aikou1} or \cite{Ich}). Then
\begin{align}
\begin{split}
\b{v}^l\frac{\delta}{\delta\b{z}^l}G_{ij}&=\b{v}^l(G_{ij;\b{l}}-G_{k;\b{l}}G^{k\b{t}}G_{ij\b{t}})\\
&=\b{v}^l\frac{\p}{\p v^j}(G_{k;\b{l}}G^{\b{t}k})G_{i\b{t}}\\
&=\b{v}^l\o{\frac{\p}{\p \b{v}^j}(G_{\b{k};l}G^{\b{k}t})}G_{i\b{t}}=0.
\end{split}
\end{align}
In particular, any flat or projectively flat Finsler metric $G$ verifies that $\frac{\p}{\p\b{v}^j}(G_{\b{l};i}G^{\b{l}k})=0$ (see \cite[Corollary 2.2, Proposition 2.18]{Aikou1}),
and so one has $\b{\p}\hat{G}(\o{\hat{P}})=0.$ 
\end{ex}

\section{Schwarz lemma for complex Finsler manifolds}
In this section, we get an inequality for complex Finsler manifolds which only contains holomorphic sectional curvature terms. By using this inequality and  the Almost Maximum Principle of Omori and Yau for complete manifolds, we obtain a generalization of the Yau's Schwarz lemma.

\begin{prop}\label{Prop1}
Let $M$ be a complex manifold of dimension $n$. Let $G_1$ and $G_2$ be two complex Finsler metrics on $M$. For any  point $(z,[v])\in P(TM)$, then
\begin{align}\label{2.1}
\frac{2}{G_2}\sup_{\pi_*(X)=v}\left\{\p\b{\p}\log\frac{G_1}{G_2}(X,\o{X})\right\}\geq K_{G_2}-\frac{G_1}{G_2}K_{G_1},
\end{align}
where $\pi_*: T(TM)\to TM$ is the differential of $\pi: TM\to M$.
\end{prop}
\begin{proof}
From Definition \ref{defn1}, one has at the point $(z,[v])\in P(TM)$, 
\begin{align}
\begin{split}
\left(K_{G_2}-\frac{G_1}{G_2}K_{G_1}\right)(z,[v])&=\sup_{\varphi}\{K_{\varphi^*G_2}(0)\}-\frac{G_1(v)}{G_2(v)}\sup_{\varphi}\left\{K_{\varphi^*G_1}(0)\right\}\\
&=\sup_{\varphi}\{K_{\varphi^*G_2}(0)\}-\sup_{\varphi}\left\{\frac{G_1(v)}{G_2(v)}K_{\varphi^*G_1}(0)\right\}\\
&\leq \sup_{\varphi}\left\{K_{\varphi^*G_2}(0)-\frac{G_1(v)}{G_2(v)}K_{\varphi^*G_1}(0)\right\}\\
&=\sup_{\varphi}\left\{\frac{2}{\varphi^*G_2}\frac{\p^2}{\p w\p\b{w}}\log \frac{\varphi^* G_1}{\varphi^* G_2}\right\}\\
&=\sup_{\varphi}\left\{\frac{2}{G_2(\lambda v)}\p\b{\p}\log \frac{G_1}{G_2}\left((\varphi_*)_*(\frac{\p}{\p w}),\o{(\varphi_*)_*(\frac{\p}{\p w}})\right)\right\}\\
&\leq \frac{2}{G_2(v)}\sup_{\pi_*(X)=v}\left\{\p\b{\p}\log \frac{G_1}{G_2}(X,\o{X})\right\}.
\end{split}
\end{align}
where the supremum is taken over all the holomorphic map $\varphi:\Delta\to M$ with $\varphi(0)=z$, $\varphi'(0)=\lambda v$ for some $\lambda\in \mb{C}^*$.
\end{proof}

\begin{rem}\label{rem1}
From the proof of above proposition, if $G_1, G_2$ only satisfy  ${\bf F1)}, {\bf F3)}$ and nonnegative, then (\ref{2.1}) also holds outside the zero points of $G_1G_2$.
\end{rem}

\begin{rem}
If $G_1$ and $G_2$ are strongly pseudoconvex complex Finsler metrics on $M$, then there is another upper-bound for $K_{G_2}-\frac{G_1}{G_2}K_{G_1}$. Denote by $\left(\frac{\delta}{\delta z^i}\right)_{\alpha}$ the horizontal lifting of $\frac{\p}{\p z^i}$ with respect to $G_\alpha, \alpha=1,2$. For any smooth function $f$ on $TM$, the horizontal Laplacian is defined by
$$\Delta^\mc{H}_{G_2}f:=\frac{v^i\b{v}^j}{G_2}(\p\b{\p}f)\left(\left(\frac{\delta}{\delta z^i}\right)_2,\left(\frac{\delta}{\delta \b{z}^j}\right)_2\right).$$
From (\ref{defn22}), one has
	\begin{align*}
		&\quad 2\Delta^{\mc{H}}_{G_2}\log \frac{G_1}{G_2}=K_{G_2}+\frac{2}{G_2}(\p\b{\p}\log G_1)\left(v^i\left(\frac{\delta}{\delta z^i}\right)_2,\o{v^i\left(\frac{\delta}{\delta z^i}\right)_2}\right)\\
		&=K_{G_2}+\frac{2}{G_2}(\p\b{\p}\log G_1)\left(v^i\left(\frac{\delta}{\delta z^i}\right)_1+v^i\left(\left(\frac{\delta}{\delta z^i}\right)_2-\left(\frac{\delta}{\delta z^i}\right)_1\right),\o{v^i\left(\frac{\delta}{\delta z^i}\right)_1+v^i\left(\left(\frac{\delta}{\delta z^i}\right)_2-\left(\frac{\delta}{\delta z^i}\right)_1\right)}\right)\\
		&=K_{G_2}-\frac{G_1}{G_2}K_{G_1}+\frac{2}{G_2}(\p\b{\p}\log G_1)\left(v^i\left(\left(\frac{\delta}{\delta z^i}\right)_2-\left(\frac{\delta}{\delta z^i}\right)_1\right),\o{v^i\left(\left(\frac{\delta}{\delta z^i}\right)_2-\left(\frac{\delta}{\delta z^i}\right)_1\right)}\right)\\
		&\geq K_{G_2}-\frac{G_1}{G_2}K_{G_1},
	\end{align*}
	where the last inequality because $G_1$ is strongly pseudoconvex.
\end{rem}

As an application of the above proposition, we give a direct proof of the following Schwarz Lemma.
\begin{cor}[{\cite[Theorem 4.1]{Shen}}]
	Let $f$ be a holomorphic map between two complex Finsler manifolds $(M, G)$ and $(N,G_1)$ with $M$ compact. If $K_{G_1}\leq -A, K_{G}\geq -B$ for  $A>0$, $B\geq 0$, then
	\begin{align}\label{1.3}
	f^*G_1\leq \frac{B}{A}G.
	\end{align}
\end{cor}
\begin{proof}
Set
$$u=\frac{f^*G_1}{G},$$
which is a smooth function on $P(TM)$. If $\max_{(z,[v])\in P(TM)}u(z,[v])=0$, then (\ref{1.3}) holds obviously. Now we assume that
$$\max_{(z,[v])\in P(TM)}u(z,[v])=u(z_0,[v_0])>0.$$
By Remark \ref{rem1}, one has
\begin{align}
0\geq \frac{2}{G(v_0)}\sup_{\pi_*(X)=v_0}\left\{\p\b{\p}\log u(X,\o{X})\right\}\geq (K_{G}-uK_{f^*G_1})(z_0,[v_0]).
\end{align}
Note that
\begin{align}\label{1.5}
K_{f^*G_1}([v])&=\sup_{\varphi'=\lambda v}\{K_{\varphi^*f^*G_1}(0)\}\leq \sup_{\psi'=\lambda f_*(v)}\{K_{\psi^*G_1}(0)\}=K_{G_1}([f_*(v)])=(f^*G_1)[v].
\end{align}
So
\begin{align}
(K_{G}-u f^*K_{G_1})(z_0,[v_0])\leq 0.
\end{align}
By the assumptions for $K_{G_1}$ and $K_{G_2}$, one obtains
\begin{align}
	f^*G_1\leq \frac{B}{A}G.
	\end{align}
\end{proof}

Next we consider the case of a holomorphic map from a complete Riemann surface to a Finsler manifold with negative holomorphic sectional curvature. We assume that $(M,G)$ is a complete Riemann surface, the fundamental form is
\begin{align}
\Phi=\sqrt{-1}\lambda dz\wedge d\b{z},\quad \lambda=\frac{1}{2}G_{1\b{1}}.
\end{align}
and $ds^2=2\lambda dz\otimes d\b{z}$. The holomorphic sectional curvature of $G$ is
\begin{align}
K_G=-\frac{1}{\lambda}\frac{\p^2}{\p z\p\b{z}}\log\lambda,
\end{align}
which is exactly the Gaussian curvature.

Now we have the following generalization of Yau's Schwarz lemma \cite[Theorem 2']{Yau1}

\begin{thm}\label{thm1}
Let $(M,G)$ be a complete Riemann surface with curvature bounded from below by a constant $K_1$. Let $(N, G_1)$ be another Finsler manifold with holomorphic sectional curvature bounded from above by a negative constant $K_2$. Then for any holomorphic map $f$ from $M$ to $N$,
\begin{align}\label{1.4}
\frac{f^*G_1}{G}\leq \frac{K_1}{K_2}.
\end{align}
\end{thm}
\begin{proof}
Set $$u=\frac{f^*G_1}{G}.$$
Since $\dim M=1$, one has that $P(TM)\simeq M$ and hence $u$ is a smooth function on $M$.
For the case of $\max_{z\in M} u(z)=0$, the inequality (\ref{1.4}) is obviously. So one may assume that $\max_{z\in M} u(z)=u(z_0)>0$ for some $z_0\in M$.

Let $\varphi(t)=(1+t)^{-1/2},$ $t>0$, then
\begin{align}
\varphi'(t)=-\frac{1}{2(1+t)^{\frac{3}{2}}}<0,\quad \varphi''(t)=\frac{3}{4(1+t)^{\frac{5}{2}}}>0.
\end{align}
Then applying the Almost Maximum Principle of Omori and Yau (\cite{Yau2, Omori}, \cite[\S 6.2]{Kim}) to $-\varphi\circ u$, there exists a sequence $\{p_{\nu}\in M|\nu=1,2,\ldots\}$ such that
\begin{align}\label{Almost}
\inf_M\varphi\circ u=\lim_{\nu\to\infty}\varphi\circ u(p_{\nu}),\quad \lim_{\nu\to\infty}\n(\varphi\circ u)|_{p_{\nu}}=0, \quad\liminf_{\nu\to\infty}\Delta(\varphi\circ u)|_{p_{\nu}}\geq 0.
\end{align}
The first equation of (\ref{Almost}) implies that
\begin{align}
\sup_M u=\lim_{\nu\to\infty}u(p_{\nu}).
\end{align}
From (\ref{1.5}) and Proposition \ref{Prop1}, one has
\begin{align}\label{1.666}
\begin{split}
\Delta u&=\frac{2}{\lambda}\frac{\p^2 u}{\p z\p\b{z}}=\frac{2}{\lambda}\left(u\frac{\p^2\log u}{\p z\p\b{z}}+\frac{1}{u}\left|\frac{\p u}{\p z}\right|^2\right)\\
&\geq 2u(K_G-uK_{f^*{G_1}})\geq 2u(K_G-uf^*K_{{G_1}})\\
&\geq 2u(K_1-uK_2).
\end{split}
\end{align}
Using the fact that
\begin{align}
\liminf_{\nu\to\infty}\Delta(\varphi\circ u)|_{p_{\nu}}=\liminf_{\nu\to\infty}\left(\varphi''(u(p_{\nu}))\|\n u|_{p_{\nu}}\|^2+\varphi'(u(p_{\nu}))\Delta u|_{p_{\nu}}\right),
\end{align}
and combining (\ref{Almost}), (\ref{1.666}) and $\varphi'(t)<0$, one gets for any $\epsilon>0$, there exists $N>0$ such that at each $p_{\nu}$ with $\nu\geq N$,
\begin{align}
2u(K_1-uK_2)\varphi'(u)+\varphi''(u)\|\n u\|^2>-\epsilon, \quad (\varphi'(u))\|\n u\|^2=\|\n(\varphi\circ u)\|^2<\epsilon^2.
\end{align}
Therefore, 
\begin{align}
K_1-uK_2<\frac{1}{2u}\left(\frac{\epsilon}{|\varphi'(u)|}+\frac{\epsilon^2\varphi''(u)}{|\varphi'(u)|^3}\right).
\end{align}
So $\sup_M u$ is bounded because the left-hand side is $O(u(p_{\nu}))$ and right-hand side is $O(\epsilon(u(p_{\nu}))^{\frac{1}{2}}+\epsilon^2u(p_{\nu}))$ as $\nu\to\infty$. Taking $\nu\to\infty$ and $\epsilon\to 0$, one gets
\begin{align}
\sup_M u\leq \frac{K_1}{K_2}.
\end{align}
\end{proof}

For a complex manifold, there is a canonical Kobayashi metric $k_M$.
\begin{defn}[\cite{Aikou1, Ko5}]\label{defn2}
The \textsl{Kobayshi metric} $k_M: TM\to \mb{R}$ of a complex manifold $M$ is defined by
$$k_M(z,v):=\inf_\varphi\left\{\frac{1}{r}; \exists f\in \text{Hom}(\Delta(r),M), f(0)=z, f'(0)=v\right\}$$
for any $(z,v)\in TM$, where the infinimum is taken for all holomorphic map $f$ from  disc $\Delta(r)$ of radius $r$ to $M$ with $f(0)=z$ and $f'(0)=v$.

For any tangent vector  $V$ on $M$, there exists a $(1,0)$-type vector $X$ such that $V=X+\b{X}=2\text{Re}X$, $k_M(V):=2k_M(X)$.
The \textsl{Kobayashi pseudo-distance} $d^K_M(p,q)$ is defined by
$$d^K_M(p,q)=\inf_c\int_0^1 k_M(c(t),c'(t))dt.$$
where infinimum is taken all curve of $C^1$-class on $M$.

A complex manifold $M$ is said to be \textsl{Kobayashi hyperbolic} if the pseudo-distance $d^K_M$ is the distance in the strict sense.
\end{defn}

 As a simple application, we have the following interesting result.
\begin{cor}[\cite{Aikou1, Ko1}]
Let $(N,G_1)$ be a complex Finsler manifold with holomorphic sectional curvature bounded from above by a negative constant $K_2$. Then we have
\begin{align}
4k^2_N\geq -K_2 G_1,
\end{align}
where $k_N$ is the Kobayashi metric of $N$. In particular, $N$ is Kobayashi hyperbolic.
\end{cor}
\begin{proof} We equip the $r$-disc $\Delta(r)$ with the complete Poincar\'e metric $G=2\lambda dz\otimes d\b{z}$, $\lambda=\frac{r^2}{2(r^2-|z|^2)}$. 
It is well known that its Gaussian curvature is $K_{G}=-4<0$. By Theorem \ref{thm1}, one has
$$f^*G_1\leq \frac{-4}{K_2}G.$$
Therefore,
$$-K_2G_1(f(0),f'(0))\leq \frac{4r^2}{(r^2-|z|^2)^2}|_{z=0}=\frac{4}{r^2}.$$
By the definition of $k_N$, one has
\begin{align}\label{1111}4k_N^2\geq -K_2G_1.\end{align}
From (\ref{1111}) and $K_2<0$, $d^K_N$ is a distance in the strict sense.
By Definition \ref{defn2}, $N$ is Kobayashi hyperbolic.
\end{proof}

\section{Semi-positive holomorphic sectional curvature}
In this section, we assume that the holomorphic sectional curvature of a strongly pseudoconvex complex Finsler manifold $(M,G)$ is semi-positive but not identically zero, i.e. $K_G\geq 0$ and there exists a point $(z_0,[v_0])\in P(TM)$ such that $K_G(z_0,[v_0])>0$. Firstly, we review Berndtsson's curvature formula of direct image bundles (cf. \cite{Bern2, Bern3, Bern4, Ke1, Liu1}).
\subsection{Curvature of direct image bundles}

Let $\pi: \mc{X}\to M$ be a holomorphic fibration with compact fibers, $L$ a relative ample line bundle over $\mc{X}$, i.e. there eixsts a metric (weight) $\phi$ of $L$ such that $\sqrt{-1}\p\b{\p}\phi|_{\mc{X}_z}>0$ for any $z\in M$, $\mc{X}_z:=\pi^{-1}(z)$. We denote by $(z;w)=(z^1,\cdots, z^n; w^1,\cdots, w^m)$ a local admissible holomorphic coordinate system of $\mc{X}$ with $\pi(z;w)=z$, where $m$ denotes the dimension of fibers.
For any smooth function $\phi$ on $\mc{X}$, we denote
$$\phi_{;i}=\frac{\p \phi}{\p z^{i}},\quad \phi_{;\b{j}}=\frac{\p \phi}{\p \b{z}^{j}},
\quad \phi_{\alpha}=\frac{\p \phi}{\p w^\alpha},\quad \phi_{\b{\beta}}=\frac{\p \phi}{\p \b{w}^\beta},$$
where $1\leq i,j\leq n, 1\leq \alpha,\beta\leq m$.

For any smooth metric $\phi$ of $L$ with $\sqrt{-1}\p\b{\p}\phi|_{\mc{X}_z}>0$, set
\begin{align}\label{horizontal}
  \frac{\delta}{\delta z^{i}}:=\frac{\p}{\p z^{i}}-\phi_{\b{\beta};i}\phi^{\b{\beta}\alpha}\frac{\p}{\p w^{\alpha}},
\end{align}
where $(\phi^{\b{\beta}\alpha})$ is the inverse matrix of $(\phi_{\alpha\b{\beta}})$. By a routine computation, one shows easily that $\{\frac{\delta}{\delta z^{i}}\}_{1\leq i\leq n}$ spans a well-defined horizontal subbundle of $T\mc{X}$. Let $\{dz^{i};\delta w^\alpha\}$
denote the dual frame of $\left\{\frac{\delta}{\delta z^{i}}; \frac{\p}{\p w^\alpha}\right\}$. One has
$$\delta w^\alpha=dw^\alpha+\phi^{\alpha\b{\beta}}\phi_{\b{\beta};i}dz^{i}.$$
Denote
\begin{align}\label{HV}
\p^V=\frac{\p}{\p w^{\alpha}}\otimes \delta w^{\alpha},\quad \p^H=\frac{\delta}{\delta z^{i}}\otimes dz^{i}.
\end{align}
Clearly, the operators $\p^V$ and $\p^H$ are well-defined.

The geodesic curvature $c(\phi)$ of $\phi$ (\cite[Definition 2.1]{Choi}) is defined by
\begin{align}\label{cphi}
  c(\phi)=c(\phi)_{i\b{j}}\sqrt{-1} dz^{i}\wedge d\b{z}^{j}=\left(\phi_{;i\b{j}}-\phi_{\b{\beta};i}\phi^{\alpha\b{\beta}}\phi_{\alpha;\b{j}}\right)\sqrt{-1} dz^{i}\wedge d\b{z}^{j},
\end{align}
which is clearly a horizontal real $(1,1)$-form on $\mc X$. By a direct computation, one has
\begin{align}
\sqrt{-1}\p\b{\p}\phi=c(\phi)+\sqrt{-1}\phi_{\alpha\b{\beta}}\delta w^\alpha\wedge \delta \b{w}^\beta.
\end{align}

We consider the direct image sheaf $$E:=\pi_*(K_{\mc{X}/M}+L).$$ Then $E$ is a holomorphic vector bundle. In fact, for any point $p\in M$, taking a local coordinate neighborhood $(U;\{z^{i}\})$ of $p$, then $\phi+\beta\sum_{i=1}^n |z^i|^2$ is a metric on the line bundle $L\to \mc{X}|_U$, whose curvature is
$$\sqrt{-1}\p\b{\p}\phi+\beta\sqrt{-1}\sum_{i=1}^n dz^{i}\wedge d\b{z}^{i}.$$
By taking $\beta$ large enough,  the curvature of $\phi+\beta\sum_{i=1}^n |z^i|^2$ is positive. By the same argument as in \cite[\S 4, page 542]{Bern2},  there exists a local holomorphic frame for $E$. So $E$ is a holomorphic vector bundle.


Following Berndtsson (cf. \cite{Bern2, Bern3, Bern4}), we define the following $L^2$-metric on
the direct image bundle $E:=\pi_*(K_{\mc{X}/M}+L)$. Let $\{u^A\}_{1\leq A\leq\text{rank}E}$ be a local frame of $E$. Set
 \begin{align}\label{L2 metric}
 h^{A\b{B}}=\langle u^A, u^B\rangle=\int_{\mc{X}_z}u^A\o{u^B}e^{-\phi}.
 \end{align}

Note that $u^A$ can be written locally as 
\begin{align}
	u^A=f^A dw\otimes s_L
\end{align}
 where $s_L$ is a local holomorphic frame of $L$ with $e^{-\phi}=|s_L|^2$,  and so locally
$$u^A\o{u^B}e^{-\phi}:=(\sqrt{-1})^{m^2}f^A\o{f^B} |s_L|^2dw\wedge d\b{w}=(\sqrt{-1})^{m^2}f^A\o{f^B} e^{-\phi}dw\wedge d\b{w},$$
where $dw=dw^1\wedge \cdots\wedge dw^m$ is the fiber volume.

 The following theorem actually was proved by Berndtsson \cite[Theorem 1.2]{Bern4}.
\begin{thm}[{\cite[Theorem 1.2]{Bern4}}]
\label{thm4} 
Denote by $\Theta^E=\Theta^{A\b{B}}_{\quad i\b{j}}u_A\otimes \o{u_B}\otimes dz^i\wedge d\b{z}^j$ the Chern curvature of the $L^2$-metric (\ref{L2 metric}) on $E$, where $\{u_A\in E^*\}$ is the dual frame of $\{u^A\}$.
Then at the point $z\in M$, 
\begin{align}\label{cur}
\Theta^{A\b{B}}_{\quad i\b{j}}=\int_{\mc{X}_z}c(\phi)_{i\b{j}}u^A\o{u}^B e^{-\phi}+\langle(1+\Box')^{-1}i_{\b{\p}^V\frac{\delta}{\delta z^{i}}}u^A,i_{\b{\p}^V\frac{\delta}{\delta z^{j}}}u^B\rangle.
\end{align}
 Here $\Box'=\n'\n'^*+\n'^*\n$ denotes the Laplacian on $\mc{X}_z$ associated with the $(1,0)$-part of the Chern connection on $L|_{\mc{X}_z}$, $\b{\p}^V\frac{\delta}{\delta z^{i}}=\frac{\p (-\phi_{\b{\beta};i}\phi^{\b{\beta}\alpha})}{\p\b{w}^\gamma}\delta\b{w}^{\gamma}\otimes \frac{\p}{\p w^{\alpha}}$.
 
Moreover, \begin{align}\label{wan111}
 	\langle(1+\Box')^{-1}i_{\b{\p}^V\frac{\delta}{\delta z^{i}}}u^A,i_{\b{\p}^V\frac{\delta}{\delta z^{j}}}u^B\rangle a_A^i\o{a_B^j}=\langle(1+\Box')^{-1}(i_{\b{\p}^V\frac{\delta}{\delta z^{i}}}u^A a_A^i),(i_{\b{\p}^V\frac{\delta}{\delta z^{j}}}u^B a_B^j)\rangle \geq 0
 \end{align}
for any element $a=a^{i}_A u^A\otimes\frac{\p}{\p z^i}\in E\otimes TM$, and 
the equality holds if and only if $$\frac{\p (\phi_{\b{\beta};i}\phi^{\b{\beta}\alpha})}{\p\b{w}^\gamma}f^A a^i_A=0.$$
\end{thm}
\begin{proof} 
For  the proof of (\ref{cur}), one also can refer to the proof of  \cite[Theorem 3.1]{FLW1}.  The equality of (\ref{wan111}) holds if and only if 
$$i_{\b{\p}^V\frac{\delta}{\delta z^{i}}}u^A a_A^i=0.$$
From (\ref{horizontal}) and (\ref{HV}), the above equation is equivalent to 
 $\frac{\p (\phi_{\b{\beta};i}\phi^{\b{\beta}\alpha})}{\p\b{w}^\gamma}f^A a^i_A=0$.
\end{proof}
\subsection{Application on cotangent bundles}

In this subsection, we will apply Theorem \ref{thm4} to the cotangent bundle $E=T^*M$. In this case, $\mc{X}=P(TM)$,
\begin{align}
E=\pi_*(\mc{O}_{P(TM)}(1))=\pi_*(L+K_{\mc{X}/M})
\end{align}
where 
\begin{align}\label{2.11}
	L:=\mc{O}_{P(TM)}(1)-K_{\mc{X}/M}=(n+1)\mc{O}_{P(TM)}(1)+\pi^*\det TM
\end{align}
(see \cite[Proposition 2.2]{Kob1}). In this case, $\text{rank} E=\dim M=n$, the dimension of fibers $m=\dim \mc{X}_z=n-1$ for any $z\in M$. 

Recall that $(z;v)$ is a local coordinate system of $TM$ with respect to the natural frame $\left\{\frac{\p}{\p z^i}\right\}_{i=1}^n$, it gives a local holomorphic frame of $\mc{O}_{P(TM)}(-1)$ by 
\begin{align}
s_{\mc{O}_{P(TM)}(-1)}(z,[v])=\frac{1}{v^k}\sum_{i=1}^n v^i\pi^*\frac{\p}{\p z^i}
\end{align}
 on $U_k:=\{(z,[v])\in P(TM)|v^k\neq 0\}$.
 
Moreover, the natural projection 
\begin{align}
q: (TM)^o\to P(TM)\quad (z; v)\mapsto (z;[v]):=(z^1,\cdots, z^n; [v^1,\cdots, v^n]),	
\end{align}
gives a local coordinate system of $P(TM)$ by 
\begin{align}\label{2.4}
(z;w)=\left(z^1,\cdots,z^n; \frac{v^1}{v^k},\cdots,\frac{v^{k-1}}{v^k},\frac{v^{k+1}}{v^{k}},\cdots, \frac{v^n}{v^k}\right)
\end{align}
on $U_k$.

Let $s_{\mc{O}_{P(TM)}(1)}$ denote the dual local holomorphic section of $s_{\mc{O}_{P(TM)}(-1)}$. By (\ref{2.11}), there exist local frame $s_L$ of $L$ and local frame $s_{\pi^*\det TM}$ of $\pi^*\det TM$ on $U_k$ such that
\begin{align}\label{2.2}
s_{\mc{O}_{P(TM)}(1)}=dw\otimes s_L,\quad  s_L=s_{\mc{O}_{P(TM)}(-1)}^{\otimes-(n+1)}\otimes s_{\pi^*\det TM}.	
\end{align}
Let $G$ be a strongly pseudoconvex complex Finsler metric on $M$, it induces a metric on $\mc{O}_{P(TM)}(-1)$ by $$|s_{\mc{O}_{P(TM)}(-1)}|^2=\frac{1}{|v^k|^2}G_{i\b{j}}(z,v)v^i\b{v}^j=\frac{G(z,v)}{|v^k|^2}$$ 
on $U_k$. Let $g=(g_{i\b{j}})$ be a Hermitian metric on $M$. From (\ref{2.2}), there exists a metric $\phi_L$ on $L$ by
\begin{align}\label{2.6}
e^{-\phi_L}:=|s_L|^2:=|s_{\mc{O}_{P(TM)}(-1)}|^{-2(n+1)}|s_{\pi^*\det TM}|^2=|v^k|^{2(n+1)}G^{-(n+1)}\det g,
\end{align}
where $\det g:=\det(g_{i\b{j}})$.

 \begin{lemma}\label{9}
On $U_{k}=\{(z,[v])\in P(TM)|v^{k}\neq 0\}$, one has
 $$\det\left(\frac{\p^2\log G}{\p w^{\alpha}\p\b{w}^{\beta}}\right)=\frac{|v^{k}|^{2n}}{G^{n}}\det G,$$
 where $\det G:=\det(G_{i\b{j}})$.
 \end{lemma}
 \begin{proof}
 Without loss of generality, we may assume that $k=1$. For any point $p\in U_1$, one has from (\ref{2.4}),
 \begin{align}\label{2.5}
 	q_*\left(\frac{\p}{\p v^1}\right)=-\sum_{i=2}^{n}\frac{v^{i}}{(v^1)^2}\frac{\p}{\p w^{i-1}},\quad  q_*\left(\frac{\partial}{\partial v^{k}}\right)=\frac{1}{v^{1}}\frac{\partial}{\partial w^{k-1}},\quad k\geq 2.
 \end{align}
Here $q_*: T(TM)^o\to TP(TM)$ denotes the differential of the natural projection $q: (TM)^o\to P(TM)$.

  Denote
 $T=v^{k}\frac{\partial}{\partial v^{k}}$ and
 $\langle \frac{\partial}{\partial v^{i}},\frac{\partial}{\partial v^{j}}\rangle:=\frac{1}{G}G_{i\b{j}}$, one has $\langle T,T\rangle=1$.
 With respect to the basis $\{T,\frac{\partial}{\partial v^{2}},\cdots,\frac{\partial}{\partial v^{n}}\}$, the matrix of $\langle\cdot,\cdot\rangle$ is
 $B(\frac{1}{G}G_{i\b{j}})\bar{B}^{T}$,
 where
 \begin{equation*}
    B=\begin{bmatrix}
    v^{1} & v^{2} & \dots & v^{n}\\
    0 & 1 & \dots & 0\\
    \vdots & \vdots &\ddots & \vdots\\
    0 &0 &\dots &1
    \end{bmatrix}.
 \end{equation*}

Denote $(\frac{\partial}{\partial v^{k}})^{\perp}=\frac{\partial}{\partial v^{k}}-\langle \frac{\partial}{\partial v^{k}},T\rangle T$. 
  For any $i,j\geq 2$, we have from (\ref{2.5})
 \begin{align*}
 \left\langle (\frac{\partial}{\partial v^{i}})^{\perp},(\frac{\partial}{\partial v^{j}})^{\perp}\right\rangle&=\frac{1}{G}G_{i\bar{j}}-\frac{1}{G^{2}}
 G_{i}G_{\bar{j}}\\
 &=\frac{\partial^{2}\log G}{\partial v^{i}\partial\bar{v}^{j}}
 =\frac{1}{|v^{1}|^{2}}\frac{\partial^{2}\log G}{\partial w^{i}\partial\bar{w}^{j}}.
 \end{align*}
 So with respect to the basis $\{T,(\frac{\partial}{\partial \zeta^{2}})^{\perp},\cdots,(\frac{\partial}{\partial \zeta^{r}})^{\perp}\}$, the matrix of $\langle\cdot,\cdot\rangle$ is
 \begin{equation*}
    \begin{bmatrix}
    \begin{BMAT}{c.c}{c.c}
    1 & 0\\
    0& \left(\frac{1}{|v^{1}|^{2}}\frac{\partial^{2}\log G}{\partial w^{\alpha}\partial\bar{w}^{\beta}}\right)_{1\leq\alpha,\beta\leq n-1}
    \end{BMAT}
    \end{bmatrix}.
 \end{equation*}
 Since $\{T,(\frac{\partial}{\partial v^{2}})^{\perp},\cdots,(\frac{\partial}{\partial v^{n}})^{\perp}\}$ and $\{T,\frac{\partial}{\partial v^{2}},\cdots,\frac{\partial}{\partial v^{n}}\}$ differ by an unitary matrix, so 
 $$\frac{1}{|v^1|^{2(n-1)}}\det\left(\frac{\p^2\log G}{\p w^{\alpha}\p\b{w}^{\beta}}\right)=\det\left(B(\frac{1}{G}G_{i\b{j}})\bar{B}^{T}\right)=|v^1|^2\frac{\det (G_{i\b{j}})}{G^{n}},$$
 which completes the proof.
 \end{proof}

Combining (\ref{2.6}) with Lemma \ref{9}, one has
\begin{align}\label{1.9}
e^{-\phi_L}=\frac{|v^k|^2\det(\p_{w^{\alpha}}\p_{\b{w}^{\beta}}\log G)}{G(\det G)(\det g)^{-1}}
\end{align}
on $U_k$.

It is known that $H^0(\mb{P}^{n-1},\mc{O}_{\mb{P}^{n-1}}(k))$ can be identified as the space of homogeneous polynomials of degree $k$ in $n$ variables. Therefore, the sections of $H^0(\mc{X}_z,\mc{O}_{P(TM)}(1)|_{\mc{X}_z})$ are of the form 
\begin{align}\label{2.3}
	u^i=v^i\left(v^j\pi^*\frac{\p}{\p z^j}\right)^*=\frac{v^i}{v^k} s_{\mc{O}_{P(TM)}(1)}=\frac{v^i}{v^k} dw\otimes s_L.
\end{align}
 
On the other hand, there is a canonical element 
\begin{align}\label{2.7}
a:=u^i\otimes \frac{\p}{\p z^i}=\left(v^i\pi^*\frac{\p}{\p z^i}\right)^*\otimes v^i\frac{\p}{\p z^i}\in E\otimes TM,
\end{align}
which is well-defined since $v^i$ and $\frac{\p}{\p z^i}$ are covariant each other. 

By (\ref{2.3}), (\ref{2.7}) and Theorem \ref{thm4}, we obtain
\begin{align}\label{1.7}
\Theta^{i\b{j}}_{~~i\b{j}}=\int_{\mc{X}_z}c(\phi)_{i\b{j}}\frac{v^i\o{v}^j}{|v^k|^2} e^{-\phi_L}(\sqrt{-1})^{(n-1)^2}dw\wedge d\b{w}+\langle(1+\Box')^{-1}i_{\b{\p}^V\frac{\delta}{\delta z^{i}}}u^i,i_{\b{\p}^V\frac{\delta}{\delta z^{j}}}u^j\rangle.
\end{align}

 Now let us compute the first term in the right-hand side of (\ref{1.7}). By (\ref{2.6}), (\ref{1.9}) and setting $\phi_G=\log G$, one has
\begin{align}\label{1.11}
\begin{split}
&\quad\int_{\mc{X}_z}c(\phi_L)_{i\b{j}}\frac{v^i\o{v}^j}{|v^k|^2} e^{-\phi_L}(\sqrt{-1})^{(n-1)^2}dw\wedge d\b{w}\\
&=\int_{\mc{X}_z}((n+1)c(\phi_G)_{i\b{j}}-\p_i\p_{\b{j}}\log \det g)\frac{v^i\b{v}^j}{G}\frac{|v^k|^{2n}\det g}{G^n}(\sqrt{-1})^{(n-1)^2}dw\wedge d\b{w}\\
&=\int_{\mc{X}_z}((n+1)c(\phi_G)_{i\b{j}}-\p_i\p_{\b{j}}\log \det g)\frac{v^i\b{v}^j}{G}\frac{\det g}{\det G}\frac{(\sqrt{-1}\p\b{\p}\log G)^{n-1}}{(n-1)!}\\
&=-(n+1)\int_{\mc{X}_z}R_{i\b{j}k\b{l}}\frac{v^i\b{v}^jv^k\b{v}^l}{G^2}\frac{\det g}{\det G}\frac{(\sqrt{-1}\p\b{\p}\log G)^{n-1}}{(n-1)!}\\
&\quad -\p_i\p_{\b{j}}\log\det g\int_{\mc{X}_z}\frac{v^i\b{v}^j}{G}\frac{\det g}{\det G}\frac{(\sqrt{-1}\p\b{\p}\log G)^{n-1}}{(n-1)!}.
\end{split}
\end{align}

Now the induced Hermitian metric on $T^*M$ is 
\begin{align}\label{1.10}
h^{\b{j}i}=\int_{\mc{X}_z}\frac{v^i\b{v}^j}{G}\frac{\det g}{\det G}\frac{(\sqrt{-1}\p\b{\p}\log G)^{n-1}}{(n-1)!}.
\end{align}

Denote by $h=(h_{i\b{j}})$  the dual Hermitian metric of $(h^{\b{j}i})$, and $\det h:=(\det(h^{\b{j}i}))^{-1}$.
Consider the induced metric on $T^*M$ given by (\ref{1.10}). We claim that by an appropriate rescaling of $g$ one can always reduce to the case in which $\det g=(\det h)^{-1}$. Indeed, in place of $g$  one can always consider the Hermitian metric
\begin{align}
\tilde{g}_{i\b{j}}=\left(\frac{\det h}{\det g}\right)^{\frac{1}{n(n+1)}}g_{i\b{j}}.
\end{align}
So there induces a Hermitian metric $\tilde{h}$ on $T^*M$  associated with the Hermitian metric $\tilde{g}=(\tilde{g}_{i\b{j}})$, which is given by (\ref{1.10}). Moreover, 
\begin{align}
\begin{split}
\det(\tilde{h}^{\b{j}i})&=\det\left(\int_{\mc{X}_z}\frac{v^i\b{v}^j}{G}\frac{\det \tilde{g}}{\det G}\frac{(\sqrt{-1}\p\b{\p}\log G)^{n-1}}{(n-1)!}\right)\\
&=\left(\frac{\det h}{\det g}\right)^{\frac{n}{n+1}}(\det h)^{-1}\\
&=\left(\left(\frac{\det h}{\det g}\right)^{\frac{1}{n+1}}\det g\right)^{-1}\\
&=(\det\tilde{g})^{-1}.
\end{split}
\end{align}

 For a Hermitian metric $h=(h_{i\b{j}})$ on $M$, the two scalar curvatures are defined by 
\begin{align}\label{1.12}
s_h=\Theta_{i\b{j}k\b{l}}h^{i\b{j}}h^{k\b{l}},\quad \hat{s}_h=\Theta_{i\b{j}k\b{l}}h^{i\b{l}}h^{k\b{j}},
\end{align}
where $\Theta_{i\b{j}k\b{l}}h^{i\b{j}}=-\frac{\p^2 h_{k\b{l}}}{\p z^i\p\b{z}^j}+h^{\b{q}p}\frac{\p h_{p\b{j}}}{\p\b{z}^l}\frac{\p h_{i\b{q}}}{\p z^k}$. 

From the chosen of Hermitian metric $g$ and (\ref{1.12}), (\ref{1.11}) reduces to 
\begin{align}\label{1.22}
\int_{\mc{X}_z}c(\phi_L)_{i\b{j}}\frac{v^i\o{v}^j}{|v^k|^2} e^{-\phi_L}(\sqrt{-1})^{(n-1)^2}dw\wedge d\b{w}=-\frac{n+1}{2}\int_{\mc{X}_z}K_G\frac{\det g}{\det G}\frac{(\sqrt{-1}\p\b{\p}\log G)^{n-1}}{(n-1)!}+s_h. 
\end{align}

Now we deal with the second term in the RHS of (\ref{1.7}). By Theorem (\ref{thm4}), $$\langle(1+\Box')^{-1}i_{\b{\p}^V\frac{\delta}{\delta z^{i}}}u^i,i_{\b{\p}^V\frac{\delta}{\delta z^{j}}}u^j\rangle \geq 0$$ and $\langle(1+\Box')^{-1}i_{\b{\p}^V\frac{\delta}{\delta z^{i}}}u^i,i_{\b{\p}^V\frac{\delta}{\delta z^{j}}}u^j\rangle=0$ if and only if 
\begin{align}\label{1.13}
\frac{\p ((\phi_L)_{\b{\beta};i}(\phi_L)^{\b{\beta}\alpha})}{\p\b{w}^\gamma}v^i=0.
\end{align}
From the definition of $\phi_L$ (\ref{2.6}) and $\det g=(\det g)(z)$, (\ref{1.13}) is equivalent to 
\begin{align}\label{1.20}
v^i\frac{\p}{\p\b{w}^{\gamma}}((\log G)_{\b{\beta};i}(\log G)^{\b{\beta}\alpha})=(\log G)^{\b{\beta}\alpha}v^i\frac{\delta_{\phi}}{\delta z^i}(\log G)_{\b{\beta}\b{\gamma}}=0,
\end{align}
where $\frac{\delta_{\phi}}{\delta z^i}$ is defined by (\ref{horizontal}). On the other hand, by definition of $\hat{G}$, $\hat{P}$ and Lemma \ref{1.111}, one has
\begin{align}
\begin{split}\label{1.21}
\frac{1}{G}(\b{\p}\hat{G})(\o{\hat{P}})&=\frac{\b{v}^l}{G}\frac{\delta}{\delta \b{z}^l}(G_{ij})\delta v^i\otimes \delta v^j\\
&=\b{v}^l\frac{\delta}{\delta \b{z}^l}((\log G)_{ij})\delta v^i\otimes \delta v^j\\
&=\b{v}^l\frac{\delta_{\phi}}{\delta \b{z}^l}((\log G)_{ij})\delta v^i\otimes \delta v^j\\
&=\b{v}^l\frac{\delta_{\phi}}{\delta \b{z}^l}\left(\frac{\p^2\log G}{\p w^{\alpha}\p w^{\beta}}\frac{\p w^{\alpha}}{\p v^i}\frac{\p w^{\beta}}{\p v^j}+\frac{\p\log G}{\p w^{\alpha}}\frac{\p^2 w^{\alpha}}{\p v^i\p v^j}\right)\delta v^i\otimes \delta v^j\\
&=\b{v}^l\frac{\delta_{\phi}}{\delta \b{z}^l}((\log G)_{\alpha\beta})\delta w^{\alpha}\otimes \delta w^{\beta},
\end{split}
\end{align}
where the third equality holds since $q_*(\frac{\delta}{\delta \b{z}^l})=\frac{\delta_{\phi}}{\delta \b{z}^l}$.

Combining (\ref{1.20}) and (\ref{1.21}), $\langle(1+\Box')^{-1}i_{\b{\p}^V\frac{\delta}{\delta z^{i}}}u^i,i_{\b{\p}^V\frac{\delta}{\delta z^{j}}}u^j\rangle=0$ if and only if the complex Finsler metric $G$ satisfies (\ref{1.2}). The curvature term $\Theta^{i\b{j}}_{~~i\b{j}}$ in (\ref{1.7}) is 
\begin{align}\label{1.23}
\Theta^{i\b{j}}_{~~i\b{j}}=-\Theta_{k\b{l}i\b{j}}h^{k\b{j}}h^{i\b{l}}=-\hat{s}_h.
\end{align}
Combining (\ref{1.7}), (\ref{1.22}) with (\ref{1.23}), one obtains
\begin{prop}\label{Prop2}
Let $G$ be a strongly pseudoconvex complex Finsler metric on complex manifold $M$. There exist  Hermitian metrics $h$ and $g$ on $M$ such that the sum of  two scalar curvatures satisfy
\begin{align}
s_h+\hat{s}_h\leq \frac{n+1}{2}\int_{\mc{X}_z}K_G\frac{\det g}{\det G}\frac{(\sqrt{-1}\p\b{\p}\log G)^{n-1}}{(n-1)!}.
\end{align} 
Moreover, the equality holds if and only if the Finsler metric satisfies (\ref{1.2}).
\end{prop}

Let 
$$\omega_h=\sqrt{-1}h_{i\b{j}}dz^i\wedge d\b{z}^j$$
be the fundament form associate to the Hermitian metric $h=(h_{i\b{j}})$ on $M$. The first application of (\ref{Prop2}) as follows.
\begin{cor}
If $K_G\leq 0$ and $M$ is compact, then there exists a Hermitian metric $h$ such that
\begin{align}
\int_M\hat{s}_h\omega_h^n\leq 0.
\end{align}
\end{cor}
 \begin{proof}
 	There is a relation between the two scalar curvatures of a Hermitian metric (see eg. \cite[Formula (3.3)]{Xiao})
 	\begin{align}
 	s_h=\hat{s}_h+\langle\b{\p}\b{\p}^*\omega_h,\omega_h\rangle.	
 	\end{align}
Integrating both sides of above equation, one gets
\begin{align}
\int_M \hat{s}_h\omega_h^n\leq 	\int_M \hat{s}_h\omega_h^n+\int_M |\b{\p}^*\omega_h|^2\omega_h^n= \int_M s_h\omega_h^n.
\end{align}
Combining Proposition \ref{Prop2} with the assumption $K_G\leq 0$, one has
\begin{align}
\int_{M}\hat{s}_h\omega_h^n\leq \frac{1}{2}\int_M (\hat{s}_h+s_h)\omega_h^n\leq \frac{n+1}{4}\int_{P(TM)}K_G\frac{\det g}{\det G}\frac{(\sqrt{-1}\p\b{\p}\log G)^{n-1}}{(n-1)!}\wedge \omega_h^n\leq 0.
\end{align}
 \end{proof}

\begin{thm}
Let $(M,G)$ be  a strongly pseudoconvex complex Finsler metric which satisfying (\ref{1.2}). If $K_G\geq 0$ and $K_G(p)>0$ for some $p\in P(TM)$, then
the Kodaira dimension  $\kappa(M)=-\infty$, i.e. $H^0(M, lK_M)=0$ for  $l\geq 1$.
\end{thm}
\begin{proof}
	By Proposition \ref{Prop2} and $(\ref{1.2})$, so there exist  Hermitian metrics $h$ and $g$ on $M$ such that
	\begin{align}
		s_h+\hat{s}_h=\frac{n+1}{2}\int_{\mc{X}_z}K_G\frac{\det g}{\det G}\frac{(\sqrt{-1}\p\b{\p}\log G)^{n-1}}{(n-1)!}.
	\end{align}
	By the assumption of $K_G$, so $s_h+\hat{s}_h$ is also semi-positive and strict positive at some point in $M$. The same proof as in (\cite[Theorem 3.1]{Xiao}, \cite[Theorem 1]{Balas}), one has $H^0(M,lK_M)=0$ for $l\geq 1$. 
\end{proof}

\end{document}